\documentclass{article}

\usepackage{epsfig}
 \usepackage{epstopdf}
\usepackage[T1]{fontenc}
\usepackage{geometry}
\usepackage{amsbsy,amsmath,latexsym,amsfonts, epsfig, color, authblk, amssymb, graphics, bm}
\usepackage{epsf,slidesec,epic,eepic}
\usepackage{fancybox}
\usepackage{fancyhdr}
\usepackage{setspace}
\usepackage{nccmath}
\usepackage[colorlinks,linkcolor=black,anchorcolor=black,citecolor=black,hyperindex,CJKbookmarks]{hyperref}
\newenvironment{proof}{{\noindent \bf Proof.}}{\hfill$\Box$\medskip}

\newtheorem{theorem}{Theorem}[section]
\newtheorem{corollary}[theorem]{Corollary}
\newtheorem{lemma}[theorem]{Lemma}

\newtheorem{remark}{Remark}

\def \b{\beta}

\def \G{\Gamma}

\def \d{\delta}

\def \w{\omega}

\def \S{\Sigma}

\def \R{\mathbb{R}}
\def \Z{\mathbb{Z}}
\def \N{\mathbb{N}}

\def \A{\mathcal{A}}

\begin{document}

\title{ The exceptional sets on the run-length function  of beta-expansions\footnotetext {* Corresponding author}
}
\author{  Lixuan Zheng$^\dag$, Min Wu$^\dag$ and Bing Li$^{\dag, *}$\\
\small \it $\dag$ Department of Mathematics\\
\small \it South China University of Technology\\
\small \it Guangzhou 510640, P.R. China\\
\small \it E-mails: z.lixuan@mail.scut.edu.cn, wumin@scut.edu.cn and
scbingli@scut.edu.cn}
\date{5th March,2015}
\date{}
\maketitle
\begin{center}
\begin{minipage}{120mm}{\small {\bf Abstract.} Let $\beta > 1$ and the run-length function $r_n(x,\beta)$ be the maximal length of consecutive zeros amongst the first n digits in the $\beta$-expansion of $x\in[0,1]$. The exceptional set $$E_{\max}^{\varphi}=\left\{x \in [0,1]:\liminf_{n\rightarrow \infty}\frac{r_n(x,\beta)}{\varphi(n)}=0, \limsup_{n\rightarrow \infty}\frac{r_n(x,\beta)}{\varphi(n)}=+\infty\right\}$$ is investigated,  where  $\varphi: \mathbb{N} \rightarrow \R^+$ is a monotonically increasing function  with $\lim\limits_{n\rightarrow \infty }\varphi(n)=+\infty$. We prove that the set $E_{\max}^{\varphi}$ is either empty or of full Hausdorff dimension and residual in $[0,1]$ according to the increasing rate of $\varphi$ .}
\end{minipage}
\end{center}

\vskip0.5cm

\section{Introduction}
For a real number $\b > 1$. Let $T_{\b}:(0,1] \rightarrow (0,1] $ be a \emph{$\b$-transformation} which is given by $$T_{\b}x = \b x-\lceil\b x\rceil + 1,$$ where $\lceil x \rceil$ means the smallest integer which is larger than $x$. By the iteration of $T_{\beta}$ (see \cite{R}), we have already known that every $x \in (0, 1]$ can be expressed as:
\begin{equation}\label{1.1}
x=\frac{\varepsilon_1(x,\beta)}{\b}+\cdots+\frac{\varepsilon_n(x,\beta)}{\beta^n}+\cdots,
\end{equation}
where, for each $n\geq1$,$$\varepsilon_n(x,\beta)=\lceil \beta T_{\beta}^{n-1}x
\rceil-1$$ which is said to be the $n$-th digit of $x$ with respect to the base $\beta$. The digit sequence is denoted by $$\varepsilon(x,\beta):=(\varepsilon_1(x,\beta), \ldots,\varepsilon_n(x,\beta),\ldots)$$and called the \emph{$\beta$-expansion} of $x$. For the sake of simplicity, we set the $\beta$-expansion of $0$ as $\varepsilon(0,\beta)=(0,0,0,\ldots).$

For every real number $x\in [0,1]$, and every integer $n\geq 1$, the {\it run-length function}, denoted by $r_n(x,\beta)$, is the maximal length of consecutive zeros amongst $\varepsilon(x,\beta)=(\varepsilon_1(x,\beta),\ldots,\varepsilon_n(x,\beta))$, that is, $$r_n(x,\beta)=\max\{j\geq 1: \varepsilon_{i+1}(x,\beta)=\cdots=\varepsilon_{i+j}(x,\beta)=0 \ {\rm for\ \ some}\ 0\leq i \leq n-j\}.$$ We set $r_n(x,\beta)=0$ if such $j$ does not exist. For the base $\beta=2$, Erd\"{o}s and R\'{e}nyi \cite{ER} showed that for Lebesgue almost all $x\in[0,1]$, we can obtain that
\begin{equation}\label{1}
\lim_{n\rightarrow \infty}\frac{r_n(x,2)}{\log_2n}=1.
\end{equation}
The size of sets about the function $r_n(x,2)$ has raised much attention. Ma et al. \cite{MW} proved that the set of points violating (\ref{1}) is of full dimension. Moreover, Li and Wu replaced the function $\log_2 n$ in (\ref{1}) by a monotonically increasing function $\varphi: \N \rightarrow \R^+$ with $\lim\limits_{n\rightarrow \infty }\varphi(n)=+\infty,$ and they introduced the exceptional set which contains those "worst" divergence point as
$$E_{\max}=\left\{x \in [0,1]:\liminf_{n\rightarrow \infty}\frac{r_n(x,2)}{\varphi(n)}=0,\ \limsup_{n\rightarrow \infty}\frac{r_n(x,2)}{\varphi(n)}=+\infty\right\}.$$ They first got a weaker conclusion that the Hausdorff dimension of $E_{\max}$ is 1 and $E_{\max}$ is residual in $[0,1]$ when the monotonically increasing function $\varphi(n)$ satisfying $\lim\limits_{n\rightarrow \infty }\frac{n}{\varphi(n^{1+\alpha})}=+\infty$ for some $0<\alpha\leq 1$, see \cite{LW1} for more details. After that, in \cite{LW2}, they showed that the $E_{\max}$ has Hausdorff dimension 1 and is of residue under the condition  $\lim\limits_{n\rightarrow \infty }\frac{n}{\varphi(n)}=+\infty$. If we let $\varphi(n)=\log_2 n$ in $E_{\max}$, the result that  $\dim_{\rm H}E_{\max}=1$ in \cite{LW1} is somewhat surprising since this set is much smaller than the set which was studied by Ma et al. \cite{MW}. Naturally, it is of interest to consider whether the above properties will be true if  2 is substituted for a general real number $\beta>1$. As a matter of fact, Tong et al. \cite{TYZ} gave a similar result as \cite{ER} that for Lebesgue almost all $x\in[0,1]$, we have \begin{equation}\label{2}
\lim_{n\rightarrow \infty}\frac{r_n(x,\beta)}{\log_\beta n}=1.
\end{equation}
Thus, the set $E=\left\{x \in [0,1]:\liminf\limits_{n\rightarrow \infty}\frac{r_n(x,\beta)}{\log_\beta n}<\limsup\limits_{n\rightarrow \infty}\frac{r_n(x,\beta)}{\log_\beta n}\right\}$ has null Lebesgue measure. By using the monotonically increasing function $\varphi: \N \rightarrow \R^+$ with $\lim\limits_{n\rightarrow \infty }\varphi(n)=+\infty$ instead of the function $\log_\beta n$, we take the exceptional set containing the "worst" divergence points as follows into consideration, that is,
\begin{equation}\label{set}
E_{\max}^{\varphi}=\left\{x \in [0,1]:\liminf_{n\rightarrow \infty}\frac{r_n(x,\beta)}{\varphi(n)}=0,\ \limsup_{n\rightarrow \infty}\frac{r_n(x,\beta)}{\varphi(n)}=+\infty\right\}.
\end{equation} We extend Li and Wu's results (see\cite{LW1,LW2}) by generalizing the base $2$ into a real number $\beta>1$, which can be expressed as the following theorems. Although the following result is similar as Li and Wu's results, we use a different and simpler way to obtain the Hausdorff dimension of the set $E_{\max}^{\varphi}$.
\begin{theorem}\label{th1}
Let $\varphi: \N \rightarrow \R^+$ with $\lim\limits_{n\rightarrow \infty }\varphi(n)=+\infty$. Let $E_{\max}^{\varphi}$ be the set defined as (\ref{set}), then

(1) If\ $\limsup\limits_{n\rightarrow \infty}\frac{n}{\varphi(n)}<+\infty$, we have $E_{\max}^{\varphi}=\emptyset$;

(2) If\ $\limsup\limits_{n\rightarrow \infty}\frac{n}{\varphi(n)}=+\infty$, we have $\dim_{\rm{H}}E_{\max}^{\varphi}=1$.
\end{theorem}
\begin{remark}\label{re1}
The result of (1) in Theorem \ref{th1} is obvious since if $\limsup\limits_{n\rightarrow \infty}\frac{n}{\varphi(n)}<+\infty$, the fact that $r_n(x,\beta)\leq n$ for all $n\geq 1$ and $x\in[0,1]$ gives that $$\limsup_{n\rightarrow \infty}\frac{r_n(x,\beta)}{\varphi(n)}\leq \limsup\limits_{n\rightarrow \infty}\frac{n}{\varphi(n)}<+\infty.$$ Thus, $E_{\max}^{\varphi}=\emptyset$. So we only need to show (2) in Theorem \ref{th1} in this paper.
\end{remark}

It occurs naturally to know how large the set $E_{\max}^{\varphi}$ is in the topological sense which is another method of describing the size of a set. In topology, the notion of residual set is usually used to describe a set being large. In a metric space $X$, a set $R$ is said to be \emph{residual} if its complement is of the first category. We can get from \cite{JC} that in a complete metric space a set is residual if it contains a dense $G_\d$ set, that is, a countable intersection of open sets. The Baire category theorem \cite{JC} is an important tool in general topology and functional analysis. There are many results showing that there are some sets which are negligible in the sense of measure theory but can be large from the topological viewpoint, some interesting examples can be found in \cite{PT,BO,H,M,MP,O}. Now we establish the following theorem.
\begin{theorem}\label{th2}
Let $\varphi: \N \rightarrow \R^+$ with $\lim\limits_{n\rightarrow \infty }\varphi(n)=+\infty$. Let $E_{\max}^{\varphi}$ be the set defined as (\ref{set}), we have $E_{\max}^{\varphi}$ is residual when  $\limsup\limits_{n\rightarrow \infty}\frac{n}{\varphi(n)}=+\infty$.
\end{theorem}

By setting $\varphi(n)=\log_{\beta}n$ in the set  $E_{\max}^{\varphi}$ given as (\ref{set}) and combining the results of Tong et al. \cite{TYZ}, the following corollary  is immediate. This result gives an example that a set can be very small in the sense of topology but be large from the  measure-theoretical and dimensional points of view.
\begin{corollary}
The set $\{x: \lim\limits_{n\rightarrow \infty}\frac{r_n(x,\beta)}{\log_\beta n}=1\}$ is both of full measure and of the first category.
\end{corollary}

We complete this introduction by depicting the construction of this paper. In the next section, we summarize the relevant material on the $\beta$ expansions without proofs. For the third section, it is intended to motivate our investigation of the Hausdorff dimension of the set $E_{\max}^{\varphi}$. After constructing a subset $E_p$ of $E_{\max}^{\varphi}$, we get the lower bound of $\dim_{\rm H}E_p$ is $\frac{p-1}{p}$ and then we present a proof of Theorem \ref{th1}. The last section deals with the topological property of $E_{\max}^{\varphi}$ and gives a proof of Theorem \ref{th2}. In this section, it is worth pointing out that unlike the frequencies of digits and blocks investigated in the forthcoming publications \cite{PT,BO,H,M,MP,O}, the function $\frac{r_n(x,\beta)}{\varphi(n)}$ cannot be expressed by some frequencies, hence, we apply the method in \cite{ZWL} to get the residue of $E_{\max}^{\varphi}$ .

\section{Preliminaries}
In this section, we briefly sketch some of the standard facts on $\beta$-expansions and fix some notations. See \cite{BB,B,BW,AB,HF,LW,P,R} and references therein to get further properties about $\beta$-expansions.

A classical $\b$-transformation widely applied by many researchers is $$T(x):=\b x - \lfloor\b x\rfloor,\ 0 \leq x < 1, $$where $\lfloor x \rfloor$ denotes the largest integer not exceeding $x$. The transformation $T_{\b}(x)$ being adopted here is to guarantee that every real number $x \in (0, 1]$ has an infinite series of expansion, that is,\ $\varepsilon_n(x, \b) \geq 1$ for infinitely many $n \in \N$. This is ensured by the fact that $T_{\beta}(x)$ is strictly larger than $0$. Actually, the $\b$-expansions under the above two transformations coincide except at the specific points with a finite expansion under the algorithm $T(x)$.

The definition of $T_\b(x)$ gives that the $n$-th digit of $x$ verifies that $\varepsilon_n(x,\b)\in \A=\{0,\ldots,\lceil\b\rceil-1\}$ for all $n\geq 1$. What should be pointed out is that not all infinite sequences $\varepsilon \in \A^\N$ are the $\b$-expansion of some $x\in(0,1]$. Thus, here brings about the notation of \emph{$\b$-admissible sequence}.

A word $(\varepsilon_1,\ldots,\varepsilon_n)$ is said to be \emph{admissible} with respect to the base $\beta$ if there
exists an $x \in (0, 1]$ such that the $\beta$-expansion of $x$ satisfying $\varepsilon_1(x,\beta)=\varepsilon_1,\ldots,\varepsilon_n(x,\beta)=\varepsilon_n$. An infinite digit sequence $(\varepsilon_1,\ldots,\varepsilon_n,\ldots)$ is called admissible if there exists an $x \in (0, 1]$ having the $\b$-expansion as $(\varepsilon_1,\ldots,\varepsilon_n,\ldots)$.

For convenience, write $\S_\b^n$ as the family of all $\b$-admissible words with length $n$, i.e.,\ $$\S_\b^n=\{(\varepsilon_1,\ldots,\varepsilon_n)\in \A^n: \exists\ x \in (0,1],\ {\rm such\ that\ }\varepsilon_j(x,\b)=\varepsilon_j, {\rm\  for \ all }\ 1\leq j \leq n\}.$$ Let $\S_\b^\ast$ be the family  of all $\b$-admissible words with finite length, i.e.,\
$$\S_\b^\ast=\bigcup_{n=0}^\infty \S_\b^n.$$

The \emph{lexicographical order $<_{\rm{lex}}$} being endowed in the space $\A^\N$ is defined as
follows: $$(\varepsilon_1, \varepsilon_2,\ldots)<_{\rm{lex}}(\varepsilon'_1, \varepsilon'_2,\ldots)$$if there exists an integer $k \geq 1$ such that, for all $1 \leq j < k$, $\varepsilon_j=\varepsilon'_j$ but $\varepsilon_k<\varepsilon'_k$. The symbol $\leq_{\rm{lex}}$ stands for $=$ or $<_{\rm{lex}}$.

The $\b$-expansion of the unit $1$ plays a vital role not only in researching the dynamical properties of the orbit of $1$, but also in estimating the the properties about $\beta$-admissible words (\cite{BW}, see also \cite{LW,BT}).

Let$$1=\frac{\varepsilon_1^\ast}{\b}+\cdots+\frac{\varepsilon_n^\ast}{\b^n}+\cdots$$be the $\b$-expansion of the unit $1$. For each integer $n \geq 1$, define $$t_n=t_n(\b) := \max\{k \geq 0:\varepsilon_{n+1}^\ast =\cdots= \varepsilon_{n+k}^\ast=0\}.$$ And let
\begin{equation}\label{g}
\G_n = \G_n(\b) := \max_{1 \leq k \leq n}t_k(\b).
\end{equation}

Now we give some basic properties of the admissible words as the following theorem.
\begin{theorem}[Parry \cite{P}, R$\acute{e}$nyi \cite{R}]\label{P}
Given $\beta >1$.

(1)For every $n\geq 1$,$$\omega=(\omega_1,\ldots,\omega_n)\in \Sigma_\beta^n \Longleftrightarrow \sigma ^i\omega \leq_{\rm{lex}} (\varepsilon_1^\ast(1,\beta),\ldots,\varepsilon_{n-i}^\ast(1,\beta))\ for\ all\ i \geq 1,$$ where $\sigma$ is the shift transformation such that $\sigma\omega=(\omega_2,\omega_3,\ldots).$

(2)For all $n \geq 1$,$$\beta^n \leq \sharp \Sigma_\beta^n \leq \frac{\beta^{n+1}}{\beta-1},$$ where $\sharp$ is the cardinality of a finite set.
\end{theorem}

For an admissible word $(\varepsilon_1,\ldots,\varepsilon_n)$, we define the \emph{basic interval} of order $n$ denoted by $I_n(\varepsilon_1,\ldots, \varepsilon_n)$ as $$I_n(\varepsilon_1,\ldots, \varepsilon_n):= \{x \in (0,1]: \varepsilon_j(x,\b)=\varepsilon_j,\ {\rm for\ all}\  1 \leq j \leq n\}.$$ A simple fact of the basic interval $I_n(\varepsilon_1,\ldots,\varepsilon_n)$ is that it is a left-open and right-closed interval with $\frac{\varepsilon_1}{\b}+\cdots+\frac{\varepsilon_n}{\b^n}$ as its left endpoint, the detailed proofs appear in \cite{AB}. We write the length of $I_n(\varepsilon_1,\ldots,\varepsilon_n)$ as $|I_n(\varepsilon_1,\ldots,\varepsilon_n)|$.  In \cite{LW}, it is shown that $|I_n(\varepsilon_1,\ldots,\varepsilon_n)|\leq \beta^{-n}$. We denote $I_n(x,\beta)$ as the basic interval of order $n$ which contains $x$, and respectively, $|I_n(x,\beta)|$ as its length. What should be noticed is that the basic interval $I_n(x,\beta)$ depends on $\beta$. Here and subsequently, $I_n(x)$ stands for $I_n(x,\beta)$ without any ambiguity for simplicity of notation.

The notation of full intervals is of importance to get an evaluation of the length of $I_n(\varepsilon_1,\ldots,\varepsilon_n)$. In this paper, now we give the definition and state some simple facts on the full intervals. A basic interval $I_n(\varepsilon_1,\ldots,\varepsilon_n)$ is said to be \emph{full} if its length verifies $$|I_n(\varepsilon_1,\ldots,\varepsilon_n)| = \b^{-n}.$$ Respectively, we call the corresponding word of the full basic interval as a {\it full word}.

Several characterizations and properties of full intervals are established by Fan and Wang  \cite{AB} as follows.
\begin{theorem}[Fan and Wang \cite{AB}]\label{AB}
Let $\varepsilon=(\varepsilon_1,\ldots,\varepsilon_n)\in \S_\b^n$ with $n \geq 1$.

(1)The basic interval $I_n(\varepsilon)$ is a full interval. $\Longleftrightarrow$ $T_\beta^n(I_n(\varepsilon))=(0,1].$ $\Longleftrightarrow$ For any $m \geq 1$ and any $\varepsilon'= (\varepsilon'_1,\ldots,\varepsilon'_m)\in \S_\b^m$, the concatenation $\varepsilon \ast \varepsilon'=(\varepsilon_1, \ldots,\varepsilon_n,\varepsilon'_1,\ldots, \varepsilon'_m)$ is admissible.

(2) If $(\varepsilon_1,\ldots,\varepsilon_{n-1},\varepsilon'_n)$ with $\varepsilon'_n \neq 0$ is admissible, then the basic interval $I_n(\varepsilon_1,\ldots,\varepsilon_{n-1},\varepsilon_n)$ is full for any $0 \leq \varepsilon_n < \varepsilon'_n$.

(3) If $I_n(\varepsilon)$ is full, then for any $(\varepsilon'_1,\ldots,\varepsilon'_m)\in \S_\b^m$ , we obtain the following equality that $$|I_{n+m}(\varepsilon_1,\ldots, \varepsilon_n,\varepsilon'_1,\ldots,\varepsilon'_m)| = |I_n(\varepsilon_1,\ldots,\varepsilon_n)| \cdot |I_n(\varepsilon'_1,\ldots,\varepsilon'_m)|.$$
\end{theorem}

\begin{remark}\label{co}
(1) Intuitively, it can get from Theorem \ref{AB}(3) that, for any admissible words $(\varepsilon_1,\ldots,\varepsilon_n)$, $(\varepsilon'_1,\ldots,\varepsilon'_m)$, if both the intervals $I_n(\varepsilon_1,\ldots,\varepsilon_n)$ and $I_m(\varepsilon'_1,\ldots,\varepsilon'_m)$ are full, then the concatenation $I_{n+m}(\varepsilon_1,\ldots, \varepsilon_n,\varepsilon'_1,\ldots,\varepsilon'_m)$ is still full. This gives a way to construct a new full basic interval by two full basic intervals.

(2) Another direct result from Theorem \ref{AB}(2) and (3) is that for every integer $\ell\geq 1$, for all full basic interval $I_n(\varepsilon_1,\ldots,\varepsilon_n)$, the basic interval $I_{n+\ell}(\varepsilon_1,\ldots,\varepsilon_n,0^\ell)$ is full where $0^\ell$ is a word with $\ell$ zeros, i.e., $0^\ell=(\underbrace{0,\ldots,0}_\ell)$.

(3) Recall the definition of $\Gamma_n$ as (\ref{g}), combined with Theorem \ref{AB}(2), we get that the basic interval $I_{n+\G_n+1}(\varepsilon_1,\ldots,\varepsilon_n,0^{\G_n+1})$ is full for all $n\geq 1$.
\end{remark}

Moreover, the following theorem due to Bugeaud and Wang \cite{BW} will be used in this paper to estimate the number of full basic intervals.
\begin{theorem}[Bugeaud and Wang \cite{BW}]\label{no}
There is at least one full basic interval for all $n+1$ consecutive basic intervals of order $n$.
\end{theorem}

\section{ Proof of Theorem \ref{th1}}
Before proving Theorem \ref{th1}, we introduce our method to getting the Hausdorff dimension of the set $E_{\max}^{\varphi}$. By Remark \ref{re1}, we only need to consider the case that $\limsup\limits_{n\rightarrow \infty}\frac{n}{\varphi(n)}=+\infty$, we suppose this condition is true in the remainder of this paper without otherwise specified. For any sufficiently large integer $p$, we can always construct a set $E_p \subset E_{\max}^{\varphi}$ with Hausdorff dimension being larger than $\frac{p-1}{p}$. Then by letting $p\rightarrow +\infty$, the relationship between $E_p$ and $E_{\max}^{\varphi}$ gives that $E_{\max}^{\varphi}$ is of full dimension. For more details about the Hausdorff dimension, we refer the reader to \cite{FE}.
\subsection{Construction of Cantor subset $E_p$ of $E_{\max}^{\varphi}$}
Let $p\in \N,\ p> 1$. Now we are going to construct the desired set  $E_p \subset E_{\max}^{\varphi}$ whose Hausdorff dimension is larger than $\frac{p-1}{p}$. For the sake of convenience, repeated construction of full basic intervals is applied in constructing $E_p$ which satisfies the properties above. Our construction of the set $E_p$ is divided into three steps.

\textbf{\textbf{Step \uppercase\expandafter{\romannumeral1}}} Fixed $\beta>1$. Let
\begin{equation}\label{hh}
h=\min\{k\geq 2:(1,0^{k-2},1) {\rm\ is\ \beta-admissible}\}.
\end{equation} Since the $\beta$-expansion of $x\in(0,1]$ is infinite, we get that $h$ is a finite integer. Furthermore, Theorem \ref{AB}(2) implies that $I_h(1,0^{h-1})$ is full since the word $(1,0^{h-2},1)$ is $\beta$-admissible by the  definition of $h$. The facts that $\limsup\limits_{n\rightarrow \infty}\frac{n}{\varphi(n)}=+\infty$ and $\lim\limits_{n\rightarrow \infty}\varphi(n)=+\infty$ give that there exists a subsequence $\{n_k\}_{k\geq 1}\subset \N$ satisfying
\begin{equation}\label{lim1}
\lim\limits_{k\rightarrow \infty}\frac{n_k}{\varphi(n_k)}=+\infty
\end{equation}with $n_1\geq e^{h+1}$ and
\begin{equation}\label{lim2}
n_k \geq\varphi(n_k)\geq kn_{k-1}
\end{equation} for all $k\geq 2,\ k\in \N$.

Let $G_1=\{(10^{n_1-1})\}$ be a singleton. Then the basic interval $I_{n_1}(10^{n_1-1})$ is full since $n_1\geq e^{h+1}$. For every $k\geq 1$, let $d_k=\lfloor\log n_k\rfloor$. Write
\begin{equation}\label{M}
M_{d_k}=\{(\varepsilon_1,\ldots,\varepsilon_{d_k})\in \S_{\beta}^{d_k}:\ \varepsilon_1=1\ {\rm and\ } I_{d_k}(\varepsilon_1,\ldots,\varepsilon_{d_k})\ {\rm is\ full}\}.
\end{equation}
Then the choice of $n_1\geq e^{h+1}$ ensures that $d_k> h$ for every $k\geq 2$.
For any $j\in \Z^+$, let $$t_{2j}=\left\lfloor\frac{n_{2j}-n_{2j-1}}{d_{2j-1}}\right\rfloor,\  t_{2j+1}=\left\lfloor\frac{\frac{p-1}{p}n_{2j+1}-n_{2j}}{d_{2j}}\right\rfloor.$$ As a result from the choice of $n_k$ in (\ref{lim2}), we get that $t_k\geq 1$ for each $k\geq 1$. Next, for all $k\geq 1$, define
$$G_k=\{u_k=(u_{k-1}^{(1)},\ldots,u_{k-1}^{(t_k)},0^{n_{k}-n_{k-1}-d_{k-1}t_{k}}):\ u_{k-1}^{(i)}\in M_{d_{k-1}}{\rm \ for\ all\ } 1\leq i\leq t_k \}.$$
Then, we have for each $u_k\in G_k$, the length of $u_k$ satisfies that $|u_k|=n_k-n_{k-1}$. By Theorem \ref{AB}(1) and (3), we obtain that every $u_k$ in $G_k(k\geq1)$ is admissible and it is full which demonstrates that every $u_k$ in $G_k(k\geq1)$ can be concatenated by any $\beta$-admissible word. So $D_k$ can be well defined by Theorem \ref{AB}(1) as follows.
Let \begin{equation}\label{D}
D_k= \left\{(u_1,\ldots,u_k):u_i\in G_i, \ {\rm for\ all}\ 1\leq i\leq k\right\}.
\end{equation}
\textbf{\textbf{Step \uppercase\expandafter{\romannumeral2}}} Define $J_u$ for each $u\in D_k,\ k\in \Z^+$. For each $u=(u_1,\ldots,u_k)\in D_k,$ note that the length of $u$, denoted by $|u|$, satisfying that $$|u|=|u_1|+|u_2|+\cdots+|u_k|=n_1+(n_2-n_1)+\cdots+(n_k-n_{k-1})=n_k.$$ Let $$J_u=I_{n_k}(u).$$
\textbf{\textbf{Step \uppercase\expandafter{\romannumeral3}}} Finally, set $$E_p=\bigcap_{k\geq 1}\bigcup_{u \in D_k}J_u.$$

The following lemma provides a detailed exposition of showing $E_p\subset E_{\max}^{\varphi}$ for every $p\in \N,\ p> 1$.
\begin{lemma}\label{sub}
Assume that $\limsup\limits_{n\rightarrow\infty}\frac{n}{\varphi(n)}=+\infty.$ For every $p\in \N,\ p> 1$, we have $E_p\subset  E_{\max}^{\varphi}$.
\end{lemma}
\begin{proof}
Let $x \in E_p$, we shall prove that $$\liminf_{n\rightarrow\infty}\frac{r_n(x,\beta)}{\varphi(n)}=0$$ and
$$\limsup_{n\rightarrow\infty}\frac{r_n(x,\beta)}{\varphi(n)}=+\infty.$$

On the one hand, the construction of $E_p$ yields that the word $u\in D_k$ verifies the following properties:

(1)The character of $u_{k-1}^{(i)}\in M_{d_{k-1}}(1\leq i\leq t_k)$ beginning with $1$ ensures that the maximal length of zeros in every word $u\in D_k$ only appears at the tail of $u_k\in G_k$ for all $k\geq 1$;

(2) The maximal length of zeros in every word $u_{2j+1}\in G_{2j+1}$ for all $j\geq 0$ is increasing with respect to $j$. Thus, for amplitude $j\geq 1$, noticing that $$ n_{2j}-n_{2j-1}-d_{2j-1}t_{2j}+d_{2j-1}=  n_{2j}-n_{2j-1}-d_{2j-1}\left\lfloor\frac{n_{2j}-n_{2j-1}}{d_{2j-1}}\right\rfloor+d_{2j-1}\leq 2d_{2j-1},$$ we have $$r_{n_{2j}}(x,\beta)\leq \max\{n_{2j-1}-n_{2j-2}-d_{2j-2}t_{2j-1}+d_{2j-2}, n_{2j}-n_{2j-1}-d_{2j-1}t_{2j}+d_{2j-1}\}<2n_{2j-1}.$$Therefore, $$\liminf_{n\rightarrow\infty}\frac{r_n(x,\beta)}{\varphi(n)}\leq \liminf_{j\rightarrow\infty}\frac{r_{n_{2j}}(x,\beta)}{\varphi(n_{2j})}\leq \liminf_{j\rightarrow\infty}\frac{2n_{2j-1}}{\varphi(n_{2j})}\leq \lim_{j\rightarrow\infty}\frac{2n_{2j-1}}{2jn_{2j-1}}=0,$$ where the last inequality follows from (\ref{lim2}).

On the other hand, we note that there are at least $n_k-n_{k-1}-d_{k-1}t_k$ zeros in every word $u\in D_k$.  So it holds that $$r_{n_{2j+1}}(x,\beta)\geq n_{2j+1}-n_{2j}-d_{2j}t_{2j+1}\geq n_{2j+1}-n_{2j}-d_{2j}\left(\frac{\frac{p-1}{p}n_{2j+1}-n_{2j}}{d_{2j}}+1\right)>\frac{1}{p}n_{2j+1}.$$
Thus, $$\limsup_{n\rightarrow\infty}\frac{r_n(x,\beta)}{\varphi(n)}\geq \limsup_{j\rightarrow\infty}\frac{r_{n_{2j+1}}(x,\beta)}{\varphi(n_{2j+1})}\geq \lim_{j\rightarrow\infty}\frac{\frac{1}{p}n_{2j+1}}{\varphi(n_{2j+1})}=+\infty,$$ by (\ref{lim1}).
\end{proof}
\subsection{Lower bound of $\dim_{\rm H}E_p$}
When it comes to the lower bound of  $\dim_{\rm H}E_p$, we technically show that given $\beta>1,$  $$\dim_{\rm H}E_p\geq \frac{\log \overline{\beta}}{\log \beta}\frac{p-1}{p}$$ for all $1<\overline{\beta}<\beta$. We start with introducing the following modified mass distribution principle (see \cite{BW}) which is of great importance to estimate the lower bound of the Hausdorff dimension of $E_p$. For more information of the mass distribution principle, readers can refer to \cite{FE}.
\begin{lemma}[Bugeaud and Wang \cite{BW}]\label{mp}
Let $\mu$ be the Borel measure and $E$ be a Borel measurable set with positive measure. Suppose that there exists a constant $c>0$ and an integer $N\geq 1$ satisfying that for any $n\geq N$ and basic interval of order $n$ containing $x\in E$ denoted by $I_n(x)$, we have $$\mu(I_n(x))\leq c|I_n(x)|^s.$$ Then, $\dim_{\rm H}E\geq s.$
\end{lemma}

For all $k\geq 2$, recall that $M_{d_k}$ and $D_k$ are defined as (\ref{M}) and (\ref{D}) respectively. Let $$a_k:=\sharp M_{d_k}$$ and $$b_k:=\sharp D_k.$$
\begin{lemma}\label{c}
Fixed $\beta>1,$ for each $1<\overline{\beta}<\beta$, there exist integers $k(\overline{\beta})$, $c(\overline{\beta})$ relying on $\overline{\beta}$ such that for every $k>k(\overline{\beta})$, $$a_k\geq \overline{\beta}^{d_k}$$ and $$b_k\geq c(\overline{\beta})\overline{\beta}^{p_k}$$ where $$p_k= \left\{
\begin{aligned}
n_{2j}-\frac{1}{p}\sum_{i=1}^j n_{2i-1}-\sum_{i=1}^{2j-1} d_i & , &{\rm when}\ \ k=2j,\ {\rm for some}\ j \in \N; \\
\frac{p-1}{p}n_{2j+1}-\frac{1}{p}\sum_{i=1}^j n_{2i-1}-\sum_{i=1}^{2j} d_i & , &{\rm when}\ \ k=2j+1,\ {\rm for some}\ j \in \N.\\
\end{aligned}
\right.$$
\end{lemma}
\begin{proof}
We first give the lower bound of $a_k$. Recall $h$ defined as (\ref{hh}), for any $k\geq 1,$ let $$M'_{d_k}=\{(\varepsilon_1,\ldots,\varepsilon_{d_k})\in \S_{\beta}^{d_k}:\ (\varepsilon_1,\ldots,\varepsilon_h)=(1,0^{h-1}){\rm\ and}\ I_{d_k-h}(\varepsilon_{h+1},\ldots,\varepsilon_{d_k})\ {\rm is\ full}\}.$$ Then, from the comparison of the definition of $M_{d_k}$ and $M'_{d_k}$, it holds that  $M'_{d_k}\subset M_{d_k}$ which implies that $\sharp M'_{d_k}\leq \sharp M_{d_k}$. Theorem \ref{P}(2) indicates that $$\sharp \S_{\beta}^{d_k-h}\geq \beta^{d_k-h}.$$ Furthermore, note that $\sharp M'_{d_k}$ is just the number of the full words in $\S_{\beta}^{d_k-h}$. Hence, by Theorem \ref{no}, we obtain that $$a_k\geq \sharp M'_{d_k}\geq \left\lfloor\frac{\beta^{d_k-h}}{d_k-h}\right\rfloor.$$ It follows that there exists an integer $k(\overline{\beta})$ depending on $ \overline{\beta}$ such that for every $k>k(\overline{\beta})$, we have $$\left\lfloor\frac{\beta^{d_k-h}}{d_k-h}\right\rfloor\geq \overline{\beta}^{d_k}.$$ Thus,$$a_k=\sharp M_{d_k}\geq \sharp M'_{d_k}\geq \overline{\beta}^{d_k}$$ for every $k\geq k(\overline{\beta}).$

Now we estimate $b_k$. For all $j\geq \lfloor\frac{k(\overline{\beta})}{2}\rfloor+1 \triangleq k'(\overline{\beta})$, by the construction of $G_{2j}$ and $G_{2j+1}$, we have $$\sharp G_{2j}=(\sharp M_{d_{2j-1}})^{t_{2j}}\geq \overline{\beta}^{d_{2j-1}t_{2j}}\geq \overline{\beta}^{n_{2j}-n_{2j-1}-d_{2j-1}},$$ and $$\sharp G_{2j+1}=(\sharp M_{d_{2j}})^{t_{2j+1}}\geq \overline{\beta}^{d_{2j}t_{2j+1}}\geq \overline{\beta}^{\frac{p-1}{p}n_{2j+1}-n_{2j}-d_{2j}}.$$ Then it follows from the relationship between $D_k$ and $G_k$ that for each $j\geq \lfloor\frac{k(\overline{\beta})}{2}\rfloor+1$,
\begin{equation}\label{v1}
\begin{aligned}
b_{2j}=\sharp D_{2j}=\prod_{i=1}^{2j}\sharp G_i\geq \prod_{i=k'(\overline{\beta})}^{2j}\sharp G_i&\geq \overline{\beta}^{\sum\limits_{i=k'(\overline{\beta})}^j(n_{2i}-n_{2i-1}-d_{2i-1})} \overline{\beta}^{\sum\limits_{i=k'(\overline{\beta})}^j(\frac{p-1}{p}n_{2i-1}-n_{2i-2}-d_{2i-2})}\\&\geq c(\overline{\beta})\overline{\beta}^{\sum\limits_{i=1}^j(n_{2i}-n_{2i-1}-d_{2i-1})} \overline{\beta}^{\sum\limits_{i=1}^j(\frac{p-1}{p}n_{2i-1}-n_{2i-2}-d_{2i-2})}\\&=c(\overline{\beta})\overline{\beta}^{n_{2j}- \frac{1}{p}\sum\limits_{i=1}^j n_{2i-1}-\sum\limits_{i=1}^{2j-1}d_i},
\end{aligned}
\end{equation}
where $$c(\overline{\beta})=\overline{\beta}^{-\left(\sum\limits_{i=1}^{k'(\overline{\beta})}(n_{2i}-n_{2i-1}-d_{2i-1})+ \sum\limits_{i=1}^{k'(\overline{\beta})}(\frac{p-1}{p}n_{2i-1}-n_{2i-2}-d_{2i-2})\right)}.$$ Here $c(\overline{\beta})$ is a constant depending on $\overline{\beta}$. The same way as (\ref{v1}) shows that,
\begin{equation}\label{v2}
\begin{aligned}
b_{2j+1}=\sharp D_{2j+1}=\prod_{i=1}^{2j+1}\sharp G_i &\geq c(\overline{\beta})\overline{\beta}^{\frac{p-1}{p}n_{2j+1}- \frac{1}{p}\sum\limits_{i=1}^j n_{2i-1}-\sum\limits_{i=1}^{2j}d_i}.
\end{aligned}
\end{equation}

Combining (\ref{v1}) and (\ref{v2}), the proof is finished.
\end{proof}

Now we give the lower bound of $\dim_{\rm H}E_p$ as the following result.
\begin{lemma}\label{h}
For each $p\in \N,\ p>1$. The Hausdorff dimension of $E_p$ satisfies that $$\dim_{\rm H}E_p\geq \frac{p-1}{p}.$$
\end{lemma}
\begin{proof}
It suffices to show that $$\dim_{\rm H}E_p\geq \frac{p-1}{p}\frac{\log\overline{\beta}}{\log\beta}$$ for all $1<\overline{\beta}<\beta$. To complete our proof, it falls naturally into three parts.

(1) Distribute a probability measure $\mu$ supported on $E_p$. Let $$\mu([0,1])=1,\ {\rm and}\ \mu(I_{n_1}(u))=1,\ u \in D_1.$$ For all $k\geq 1,$ and $u=(u_1,\ldots,u_{k+1})\in D_{k+1}$, we set
\begin{equation}\label{mu}
\mu(I_{n_{k+1}}(u))=\frac{\mu(I_{n_k}(u_1,\ldots,u_k))}{\sharp G_k}.
\end{equation}Then define $\mu(I_{n}(x))$ for all $n_k< n<n_{k+1}$ and $x\in E_p$ as $$\mu(I_{n}(x))=\sum_{u \in D_{k+1},I_{n_{k+1}}(u)\subset I_n(x)} \mu(I_{n_{k+1}}(u)).$$  For every $x\notin E_p$, let $\mu(I_{n}(x))=0$. The Kolmogorov's consistency theorem guarantees that $\mu$ we defined above can be uniquely extended to a Borel measure supported on $E_p$.

(2) Estimate $\frac{\log \mu(I_n(x))}{\log|I_n(x)|}$ for all $x\in E_p,\ n\geq 1$. By (\ref{mu}) and Lemma \ref{c}, we get that
\begin{equation}\label{m1}
\mu(I_{n_i}(x))=\frac{1}{b_i}\leq \frac{1}{c(\overline{\beta})\overline{\beta}^{p_i}},
\end{equation}for every $i>k(\overline{\beta}),$ where $k(\overline{\beta})$ is an integer depending on $\overline{\beta}$ given in Lemma \ref{c}. For $n\geq 1$, there exists $k\geq 0$ such that $n_k< n\leq n_{k+1}$. Then, we distinguish four cases to get the lower bound of  $\mu(I_n(x))$ for all $x\in E_p,\ n\geq 1$.

Case 1. $k=2j$ and $n_{2j}+\ell d_{2j-1}\leq n<n_{2j}+(\ell+1)d_{2j-1}$ for some $0\leq \ell \leq t_{2j}-1$. Notice that the number of $I_n(x)$ containing $I_{n_{2j+1}}(u)(u\in D_{2j+1})$ is larger than $a_{2j-1}^{\ell}$. Then $$\mu(I_n(x))\leq \mu(I_{n_{2j}+\ell d_{2j-1}}(x))\leq \mu(I_{n_{2j}}(x))a_{2j-1}^{-\ell}\leq \frac{1}{c(\overline{\beta})\overline{\beta}^{p_{2j}}\overline{\beta}^{d_{2j-1}\ell}},$$ where the last inequality follows from (\ref{m1}) and Lemma \ref{c}. Moreover, Theorem \ref{AB}(3) implies that$$|I_n(x)|\geq |I_{n_{2j}+(\ell+1)d_{2j}}(x)|=\frac{1}{\beta^{n_{2j}+(\ell+1)d_{2j}}}.$$
Consequently, $$\frac{\log \mu(I_n(x))}{\log|I_n(x)|}\geq \frac{\log \overline{\beta}^{p_{2j}+d_{2j-1}\ell}+\log c(\overline{\beta})}{ \log \beta^{n_{2j}+(\ell+1)d_{2j}}}.$$ By Lemma \ref{c}, we obtain that $$\lim_{j\rightarrow\infty} \frac{\log \overline{\beta}^{p_{2j}+d_{2j-1}\ell}+\log c(\overline{\beta})}{ \log \beta^{n_{2j}+(\ell+1)d_{2j}}}=\frac{\log \overline{\beta}}{\log \beta}.$$

Case 2. $k=2j$ and $n_{2j}+t_{2j} d_{2j-1}\leq n<n_{2j+1}$. Then $$\mu(I_n(x))= \mu(I_{n_{2j+1}}(x))\leq \frac{1}{c(\overline{\beta})\overline{\beta}^{p_{2j+1}}},$$by (\ref{m1}) and Lemma \ref{c}. Moreover, Theorem \ref{AB}(3) forces that$$|I_n(x)|\geq |I_{n_{2j+1}}(x)|=\frac{1}{\beta^{n_{2j+1}}}.$$
Hence, $$\frac{\log \mu(I_n(x))}{\log|I_n(x)|}\geq \frac{\log \overline{\beta}^{p_{2j+1}}+\log c(\overline{\beta})}{ \log \beta^{n_{2j+1}}}.$$ By Lemma \ref{c}, $$\lim_{j\rightarrow\infty} \frac{\log \overline{\beta}^{p_{2j+1}}+\log c(\overline{\beta})}{ \log \beta^{n_{2j+1}}}=\frac{p-1}{p}\frac{\log \overline{\beta}}{\log \beta}.$$

Case 3. $k=2j+1$ and $n_{2j+1}+\ell d_{2j}\leq n<n_{2j+1}+(\ell+1)d_{2j}$ for some $0\leq \ell \leq t_{2j}-1$. Similar to Case 1,  it follows from (\ref{m1}) and Lemma \ref{c} that $$\mu(I_n(x))\leq \mu(I_{n_{2j+1}+\ell d_{2j}}(x))\leq \mu(I_{n_{2j+1}}(x))a_{2j}^{-\ell}\leq \frac{1}{c(\overline{\beta})\overline{\beta}^{p_{2j+1}}\overline{\beta}^{d_{2j}\ell}}.$$We further get from Theorem  \ref{AB}(3) that $$|I_n(x)|\geq |I_{n_{2j+1}+(\ell+1)d_{2j+1}}(x)|=\frac{1}{\beta^{n_{2j+1}+(\ell+1)d_{2j+1}}}.$$
As a consequence, $$\frac{\log \mu(I_n(x)}{\log|I_n(x)|}\geq \frac{\log \overline{\beta}^{p_{2j+1}+d_{2j}\ell}+\log c(\overline{\beta})}{ \log \beta^{n_{2j+1}+(\ell+1)d_{2j+1}}}.$$ By Lemma \ref{c}, we have $$\lim_{j\rightarrow\infty} \frac{\log \overline{\beta}^{p_{2j+1}+d_{2j}\ell}+\log c(\overline{\beta})}{ \log \beta^{n_{2j+1}+(\ell+1)d_{2j+1}}}=\frac{p-1}{p}\frac{\log \overline{\beta}}{\log \beta}.$$

Case 4. $k=2j+1$ and $n_{2j+1}+t_{2j+1} d_{2j}\leq n<n_{2j+2}$. Analogously, (\ref{m1}) and Lemma \ref{c} provide that$$\mu(I_n(x))= \mu(I_{n_{2j+2}}(x))\leq \frac{1}{c(\overline{\beta})\overline{\beta}^{p_{2j+2}}}.$$ Furthermore, Theorem \ref{AB}(3) gives that $$|I_n(x)|\geq |I_{n_{2j+2}}(x)|=\frac{1}{\beta^{n_{2j+2}}}.$$
Therefore, $$\frac{\log \mu(I_n(x)}{\log|I_n(x)|}\geq \frac{\log \overline{\beta}^{p_{2j+2}}+\log c(\overline{\beta})}{ \log \beta^{n_{2j+2}}}.$$ Lemma \ref{c} indicates that, $$\lim_{j\rightarrow\infty} \frac{\log \overline{\beta}^{p_{2j+2}}+\log c(\overline{\beta})}{ \log \beta^{n_{2j+2}}}=\frac{\log \overline{\beta}}{\log \beta}.$$

(3) Use the modified mass distribution principle (Lemma \ref{mp} ) to get the lower bound of $\dim_{\rm H}E_p$. By the discussion of the above four cases in (2), we immediately get that, for every $\eta>0$, there exits an integer $n_0$ such that for every $n\geq n_0$ and $x\in E_p$, we obtain that $$\mu(I_n(x))\leq |I_n(x)|^{\frac{p-1}{p}\frac{\log\overline{\beta}}{\log\beta}-\eta}.$$ Thus, it results from Lemma \ref{mp} that $$\dim E_p \geq \frac{p-1}{p}\frac{\log\overline{\beta}}{\log\beta}-\eta.$$The arbitrariness of $\eta>0$ and $1<\overline{\beta}<\beta$ demonstrates that $$\dim E_p \geq \frac{p-1}{p}.$$
\end{proof}
\subsection{Hausdorff dimension of $E_{\max}^{\varphi}$}
From the discussion above, the remainder of this section will be devoted to the proof of Theorem \ref{th1}.

\textbf{Proof of Theorem \ref{th1}} Applying Lemma \ref{sub}, it holds that $E_p\subset E_{\max}^{\varphi}$ for every $p\in \N,\ p>1$. By setting $p\rightarrow\infty$, we get that $$\dim_{\rm H} E_{\max}^{\varphi}\geq \lim_{p\rightarrow\infty} \dim_{\rm H} E_p\geq \lim_{p\rightarrow\infty} \frac{p-1}{p}=1$$ where the second inequality follows from Lemma \ref{h}. It is obvious that $\dim_{\rm H} E_{\max}^{\varphi}\leq1$. Thus, $$\dim_{\rm H} E_{\max}^{\varphi}=1.$$$\hfill\Box$

\section{Proof of Theorem \ref{th2}}
Now we turn towards the topological property of the set $E_{\max}^{\varphi}$. Due to the Baire category theorem, we just need to construct a set $U \subset [0,1]$ verifying the following conditions:

(1) $U \subset E_{\max}^{\varphi}$;

(2) $U$ is dense in $[0,1]$;

(3) $U$ is a $G_\d$ set.

Before putting the proof of Theorem \ref{th2}, we devote to constructing a set $U$ with the desired properties.  For every integer $n\geq 1,$ let $\Gamma_n$ be defined as (\ref{g}) and $h$ be given as (\ref{hh}). Since  $\limsup\limits_{n\rightarrow \infty}\frac{n}{\varphi(n)}=+\infty$ and $\varphi(n)\rightarrow +\infty$ as $n\rightarrow +\infty$, we can choose a increasing subsequence $\{n_i\}_{i\geq 1}\subset \N$ satisfying
\begin{equation}\label{lim3}
\lim\limits_{i\rightarrow \infty}\frac{n_i}{\varphi(n_i)}=+\infty
\end{equation}with $n_i-n_{i-1}>\max\{2h, i+\G_i\}$ and $\varphi(n_i)\geq (i-1)n_{i-1}$. Fix $(\varepsilon_1,\ldots,\varepsilon_k) \in \Sigma_{\beta}^k$, let
\begin{equation}\label{om}
\omega^{(k)}_i=\left\{
\begin{aligned}
(1,0^{n_{k+i}-n_{k+i-1}-1}) & , &{\rm when}\ \ i{\rm\ is\ odd}; \\
\left((1,0^{h-1})^{\lfloor \frac{n_{k+i}-n_{k+i-1}}{h}\rfloor h},0^{n_{k+i}-n_{k+i-1}-\lfloor \frac{n_{k+i}-n_{k+i-1}}{h}\rfloor h}\right) & , &{\rm when}\ \ i {\rm\ is\ even}.\\
\end{aligned}
\right.
\end{equation} for every $1\leq i\leq 2k$.
Now we define $$U:=\bigcap_{n=1}^{\infty} \bigcup_{k=n}^{\infty} \bigcup_{(\varepsilon_1,\ldots,\varepsilon_k) \in \Sigma_{\beta}^k}{\rm int} \left(I_{n_{3k}}(\varepsilon_1,\ldots,\varepsilon_k ,0^{n_k-k},\omega^{(k)}_1,\ldots,\omega^{(k)}_{2k}\right),$$ where ${\rm int}(I_{|\varepsilon|}(\varepsilon))$ denotes the interior of $I_{|\varepsilon|}(\varepsilon)$ for every $\varepsilon \in \S^\ast_\beta$.

\begin{remark}
$U$ is well defined. This is because, for all $(\varepsilon_1,\ldots,\varepsilon_k)\in \Sigma_{\beta}^k$, it follows from Remark \ref{co}(3) that the interval $I_{k+\G_k+1}(\varepsilon_1,\ldots,\varepsilon_k,0^{\G_k+1})$ is full. Since  $n_k>k+\G_k$ by the choice of $n_k$, we have the basic interval $I_{n_{k}}(\varepsilon_1,\ldots,\varepsilon_k ,0^{n_k-k})$ is full by  Remark \ref{co}(2). So the word $(\varepsilon_1,\ldots,\varepsilon_k ,0^{n_k-k})$ can concatenate any $\b$-admissible word by Theorem \ref{AB}(1). Similarly, $\omega^{(k)}_i(1\leq i \leq 2k)$ can concatenate any admissible word by the choice of $n_k$ satisfying $n_k-n_{k-1}>2h$ for all $k\geq 2$ and it is full for each $k \geq 1$ by Remark \ref{co}(2).
\end{remark}

It is obvious that $U$ is a $G_{\d}$ set since ${\rm int}(I_{|\varepsilon|}(\varepsilon))$ is open for all $\varepsilon \in \S^\ast_\beta$. So it remains to show that $U$ is a subset of $E_{\max}^{\varphi}$ and is dense in $[0,1]$.
\begin{lemma}\label{le2}
$U \subset E_{\max}^{\varphi}$.
\end{lemma}
\begin{proof}
For every $x \in U,$  by the construction of $U$, there exist infinitely many $k$, such that the $\beta$-expansion of $x$ starts with $(\varepsilon_1,...,\varepsilon_k,0^{n_k-k},\w^{(k)}_1,\ldots,\w^{(k)}_{2k})$, where $(\varepsilon_1,...,\varepsilon_k)\in \Sigma_{\beta}^k$ and $\w^{(k)}_i$ is defined as (\ref{om}) for all $1\leq i\leq 2k$. Now we are concentrating on finding out the super limit and lower limit of $\frac{r_n(x,\beta)}{\varphi(n)}$.

When $n=n_{3k}$, the construction of $U$ gives that the maximal length of zeros can only appear in the tail of $\omega^{(k)}_i(1\leq i\leq 2k)$ defined as (\ref{om}). Moreover, $n_k$ is increasing as $k$ increases. Consequently, it comes to the conclusion that, for large enough $k$, $$r_{n_{3k}}(x,\beta)\leq \max\{k+\G_k,2h,n_{3k-1}-n_{3k-2}\}\leq n_{3k-1}.$$ Thus, we have $$\liminf_{n\rightarrow \infty} \frac{r_n(x,\beta)}{\varphi(n)}\leq \liminf_{k\rightarrow \infty}\frac{r_{n_{3k}}(x,\beta)}{\varphi(n_{3k})}\leq \liminf_{k\rightarrow \infty}\frac{n_{3k-1}}{\varphi(n_{3k})}\leq  \lim_{k\rightarrow \infty}\frac{n_{3k-1}}{(3k-1)n_{3k-1}}=0.$$

When $n=n_{3k-1}$, by the observation on $U$, there are at least $n_{3k-1}-n_{3k-2}$ zeros in $\omega^{(k)}_{2k-1}$ which is defined as (\ref{om}). We therefore obtain that $$r_{n_{3k-1}}(x,\beta)\geq n_{3k-1}-n_{3k-2}.$$ As a result, we get $$\limsup_{n\rightarrow \infty} \frac{r_n(x,\beta)}{\varphi(n)}\geq \limsup_{k\rightarrow \infty}\frac{r_{n_{3k-1}}(x,\beta)}{\varphi(n_{2k-1})}\geq \lim_{k\rightarrow \infty}\frac{n_{3k-1}-n_{3k-2}}{\varphi(n_{3k-1})}= +\infty.$$

In conclusion, it immediately holds that $x\in E_{\max}^{\varphi}$, so $U \subset E_{\max}^{\varphi}$.
\end{proof}

\textbf{Proof of Theorem \ref{th2}}  We first check that the set $$\bigcup_{(\varepsilon_1,\ldots,\varepsilon_k) \in \Sigma_{\beta}^k}{\rm int} \left(I_{n_{3k}}(\varepsilon_1,\ldots,\varepsilon_k ,0^{n_k-k},\omega^{(k)}_1,\ldots,\omega^{(k)}_{2k}\right)$$ is dense in [0,1]. That is, for all real number $x\in [0,1]$ and $r > 0$, we need to find out a real number $y\in U$ satisfying $|x-y|\leq r.$ Assume that the $\b$-expansion of $x$ is $\varepsilon(x,\b)=(\varepsilon_1(x),\varepsilon_2(x),\ldots)$. Let $\ell$ be an integer such that $\b^{-\ell} \leq r$. We get that $(\varepsilon_1(x),\ldots,\varepsilon_\ell(x)) \in \S^\ell_\beta$. Then let $$y\in I_{3n_\ell}(\varepsilon_1(x),\ldots,\varepsilon_\ell(x),0^{n_\ell-\ell},\w^{(\ell)}_1,...,\w^{(\ell)}_{2\ell})$$ where $\omega^{(\ell)}_{i}$ is defined as (\ref{om}) for all $1\leq i\leq 2\ell$. Thus $$|x-y|\leq \b^{-\ell} \leq r$$ since both the $\beta$-expansions of $x$ and $y$ begin with $(\varepsilon_1(x),...,\varepsilon_\ell(x))$. Hence, the set $$\bigcup_{(\varepsilon_1,\ldots,\varepsilon_k) \in \Sigma_{\beta}^k}{\rm int} \left(I_{n_{3k}}(\varepsilon_1,\ldots,\varepsilon_k ,0^{n_k-k},\omega^{(k)}_1,\ldots,\omega^{(k)}_{2k}\right)$$ is dense in $[0,1]$.

By the Baire category theorem, we consequently have $U$ is residual in $[0,1]$. To sum up, $E_{\max}^{\varphi}$ is residual in $[0,1]$ by Lemma \ref{le2}.

$\hfill\Box$

{\bf Acknowledgement}
This work was supported by NSFC 11371148, 11411130372 and 11601161, Guangdong Natural Science Foundation 2014A030313230, and "Fundamental Research Funds for the Central Universities" SCUT 2015ZZ055.

\end{document}